\documentclass[11 pt]{article}
\usepackage{amsmath, amssymb, amsfonts, verbatim, amsthm, xcolor}
\usepackage{graphicx}
\usepackage{setspace}
\usepackage{latexsym}
\usepackage{enumerate}
\usepackage{enumitem}
\usepackage{url, hyperref, doi}
\hypersetup{citecolor=blue, linkcolor=blue, colorlinks=true}
\usepackage{longtable, float, booktabs, cancel, cite}
\usepackage[margin=1in,nohead]{geometry}
\usepackage{colonequals}
\usepackage{todonotes}

\makeatletter
\def\namedlabel#1#2{\begingroup
    #2%
    \def\@currentlabel{#2}%
    \phantomsection\label{#1}\endgroup
}
\makeatother

\newtheorem{thm}{Theorem}[section]
\newtheorem{cor}[thm]{Corollary}

\newtheorem{lem}[thm]{Lemma}

\newtheorem{prop}[thm]{Proposition}
\newtheorem{proposition}[thm]{Proposition}
\newtheorem{corollary}[thm]{Corollary}
\newtheorem{lemma}[thm]{Lemma}
\theoremstyle{definition}\newtheorem{definition}[thm]{Definition}
\theoremstyle{definition}\newtheorem{remark}[thm]{Remark}
\theoremstyle{definition}\newtheorem{example}[thm]{Example}
\theoremstyle{definition}
\theoremstyle{definition}

\def\H{{\mathcal{H}}}
\def\C{{\mathbb{C}}}
\def\R{{\mathbb{R}}}
\def\F{{\mathbb{F}}}

\def\diag{{\rm diag}\,}
\def\cA{{\mathcal A}}

\def\IR{{\mathbb{R}}}

\def\cS{{\cal S}}

\def\tr{{\rm tr}}

\newcommand{\fraka}{\mathfrak{a}}
\newcommand{\frakb}{\mathfrak{b}}
\newcommand{\frakc}{\mathfrak{c}}

\renewcommand\le{\leqslant}
\renewcommand\ge{\geqslant}

\begin{document}
\openup .7 \jot
\title{Noncommutative Distances on Graphs: An Explicit Approach via Birkhoff-James Orthogonality}
\author{Pierre Clare \and Chi-Kwong Li \and Edward Poon \and Eric Swartz}
\date{}
\maketitle


\begin{abstract} We study the problem of calculating noncommutative distances on graphs, using techniques from linear algebra, specifically, Birkhoff-James orthogonality.  A complete characterization of the solutions is obtained in the case when the underlying graph is a path.
\end{abstract}

\section{Introduction}

The notion of \emph{noncommutative distance} emerged naturally from the reformulation of the distance in Riemannian geometry in terms of Dirac operators \cite[Chap. 6]{NCG}. More precisely, if $(\mathcal{M},g)$ is a spin Riemannian manifold, the geodesic distance $\delta_g(p,q)$ between two points on $\mathcal{M}$, defined as the infimum of the length of paths from $p$ to $q$, can be expressed as:

\begin{equation}\label{eq-dist-Riemannian}
\delta_g(p,q)= \sup\left\lbrace \left|f(p)-f(q)\right|\::\:f\in \cA\:,\:\big\|[\partial_M,f]\big\|\le1\right\rbrace
\end{equation}
where $\partial_M$ is the Dirac operator associated with the spin structure on $\mathcal{M}$, acting as an unbounded self-adjoint operator on the Hilbert space $\mathfrak{H}$ of $L^2$-spinors on $\mathcal{M}$ and $\cA$ is the abelian von Neumann algebra of bounded measurable functions on $\mathcal{M}$, acting as multiplication operators on $\mathfrak{H}$.

This approach readily generalizes to noncommutative spaces.  Indeed, $(\cA,\mathfrak{H},\partial_M)$ is an example of a spectral triple in the sense of Connes and \eqref{eq-dist-Riemannian} allows one to define a noncommutative distance between pure states of a C*-algebra $A$ represented on a Hilbert space $\H$ equipped with an operator $D$ satisfying certain conditions (see \cite[Chap. 6]{NCG}) by: \begin{equation}\label{eq-dist-spectriple}
d^D(\varphi,\psi)= \sup\left\lbrace \left|\varphi(a)-\psi(a)\right|\::\:a\in A\:,\:\big\|[D,a]\big\|\le1\right\rbrace.\end{equation}

The case where $A$ is of the form $C_0^\infty(\mathcal{M})\otimes\mathfrak{K}(\H)$, where $\mathfrak{K}(\H)$ denotes the algebra of compact operators on $\H$ (possibly finite-dimensional), is of particular relevance for applications to physics, including quantum gravity (see \cite{Survey_Martinetti} for a comprehensive survey of the history of the notion as well as recent developments). The so-called \emph{almost commutative} case, where $A=C_0^\infty\otimes M_n(\C)$ already has meaningful applications, notably to the metric interpretation of the Higgs field \cite{Spectral_action, Gravity-matter-NCG} in the standard model of elementary particles.

Motivated by these applications to physics, general noncommutative distances have been studied on various discrete and finite spaces. In that context, the operator $D$ can be understood as the adjacency matrix corresponding to a graph structure, opening the door to combinatorial arguments. In the finite case, the noncommutative distance between two vertices $i$ and $j$ of the graph associated with the operator $D$ is given by the formula stated in Definition \ref{def-NC-dist}:

\begin{equation}\label{eq-NCD_intro}
d^D(i,j)=\sup\left\lbrace \left|a_i-a_j\right|\::\:A = \operatorname{diag}(a_1,\ldots,a_n)\:,\:\big\|[A,D]\big\|\le1\right\rbrace.
\end{equation}

A systematic approach to noncommutative distances on spaces with $n$ points was adopted in \cite{NC_distances}, where general formulas were obtained for $n\leq3$ and proven not to exist for $n=4$. Nevertheless, one may be able to obtain results in some cases by exploiting the properties of the underlying graphs, and study the behavior of the distance under certain modifications of the Dirac operator. For instance, the methods developed in \cite{Besnard} rely on certain graph manipulations to provide estimates for noncommutative distances and compare them to standard (weighted) geodesic distances. Despite these advances, general results remain elusive; an example of a natural question that remains open at present is that of deciding whether the noncommutative distance between two vertices in a graph increases upon removing an edge.

The purpose of this paper is to bring techniques from modern matrix theory to bear on direct calculations of noncommutative distances on certain graphs. More precisely, we use a characterization of the notion of \emph{Birkhoff-James orthogonality}, due to H. Schneider and the second-named author \cite{LS}, to reformulate the problem of calculating \eqref{eq-NCD_intro} in the case of path graphs. This method allows us to recover known results for graphs with uniform weights \cite{lattice_94_BLS,lattice_96_Atzmon} as well as explicit formulas obtained for $n\leq4$ in \cite{NC_distances}, and to provide a reduction scheme for the general problem. On the other hand, the complexity of the solution in the case of path graphs provides strong evidence of the intractability of the general problem.

The paper is organized as follows.  In Section \ref{sec-reformulation-linalg-tech}, we introduce notation, discuss the problem of noncommutative distance on graphs, make some simplifying observations, and introduce our main tool: Birkhoff-James orthogonality (Theorem \ref{T:BJ}).  In Section \ref{sect:path}, we completely characterize the solution to the noncommutative distance problem for path graphs of length $n$. The main tool is the notion of a \emph{viable vector}, introduced in Definition \ref{D:viable}, which leads to a convenient reformulation of the problem (Theorem \ref{thm:maxatviabledecomp}). Finally, in Section \ref{sect:analysis}, we apply our results to obtain explicit formulas for small values of $n$, while also noting that the number of different formulas one must use for each $n$ is related to the number of partitions of $n$ into a sum of integers all at least $2$, which gives quantitative insight into the difficulty of solving the noncommutative distance problem in full generality.

\section{Reformulation of the Problem and Preliminaries}\label{sec-reformulation-linalg-tech}

Let us fix notation and collect various technical facts that will be of use in the remainder of the paper.

\subsection{Notation and Terminology}

Throughout the paper, we make use of standard notation from matrix theory: $M_{m,n}(\F)$ denotes the space of matrices of $m$ rows and $n$ columns with entries in the field $\F$, and $M_{n,n}(\F)$ is also denoted by $M_n(\F)$ or simply $M_n$. Vectors will be treated as matrices by means of the canonical isomorphism $M_{n,1}(\F)\simeq\F^n$ and for any vector $x\in\F^n$, we will denote by $x^*$ the row vector obtained by transposing and conjugating the column vector associated with $x$ under this isomorphism.

We will denote by $E_{i,j}$ the matrix (of any size, generally made clear from the context) with $1$ as its $(i,j)$ entry and $0$ elsewhere. The symbol $\oplus$ will be used to denote direct sums of vector spaces as well as direct sums of linear maps, generally thought of as diagonal entries in block matrices.  Finally, given a matrix $M \in M_n(\C)$, we denote by $\rho(M)$ the spectral radius of $M$.

\subsection{Noncommutative Distances on Graphs}

Let $G$ be a finite simple graph with set of vertices $V=\{1,\ldots,n\}$ and set of edges $E$. A spectral triple $(\cA,\H,D)$ can be associated with $G$ as follows. Consider the C*-algebra $\cA$ of diagonal matrices in $M_n(\C)$ acting on the Hilbert space $\H=\C^V\simeq\C^n$.  The analogue of a Dirac operator in this context is simply a self-adjoint operator in  $\operatorname{End}(\H)\simeq M_n(\C)$ that is compatible with the graph structure, \textit{i.e.}  a Hermitian matrix $D=(d_{i,j})_{1\le i,j\le n}$ whose entries satisfy:

\begin{equation}\label{eq-Dirac-op}
i\neq j\:\text{and}\:\{i,j\}\notin E\quad\Rightarrow\quad d_{i,j}=0.
\end{equation}

It is worth noting that the condition expressed in \eqref{eq-Dirac-op} encodes the compatibility of the operator $D$ with the edge structure of $G$ interpreted as a discrete differential structure on $V$ (see \cite{Discrete_diff}).

\begin{remark} The matrix entries $d_{i,j}$ of the Dirac operator $D$ will generally be assumed to be nonnegative.  Condition \eqref{eq-Dirac-op} allows for $D$ to have nonzero coefficients on the diagonal of $D$, even in the absence of loops in the graph $G$.
\end{remark}

Given such a triple, the general procedure of \eqref{eq-dist-spectriple} allows one to define a \emph{noncommutative distance} between vertices of the graph, as follows.

\begin{definition}\label{def-NC-dist}
Let $G$ be a finite simple graph with set of vertices $V=\{1,\ldots,n\}$. The noncommutative distance on $G$ associated with a Dirac operator $D$ as above is given by: \begin{equation}\label{eq-def-NC-dist}
d^D(i,j)=\sup\left\lbrace \left|a_i-a_j\right|\::\:A = \operatorname{diag}(a_1,\ldots,a_n)\:,\:\big\|[A,D]\big\|\le1\right\rbrace
\end{equation} for any $i,j$ in $V$.
\end{definition}

It will also be convenient to consider the (possibly infinite) \emph{weights} associated with $D$ and defined by: \[w_{i,j}=|d_{i,j}|^{-1}.\]  These weights can be used to define a geodesic distance $\ell^D$ on the graph by considering the infimum of weighted lengths of paths connecting two given vertices.  Both the geodesic and the noncommutative distances associated with $D$ are generalized distances in the sense that they can assume infinite values. However, it is proved in \cite[Proposition 4]{NC_distances} that the noncommutative distance $d^D(i,j)$ is finite if and only if the vertices $i$ and $j$ are connected in $G$. Another result of \cite{NC_distances} shows that the noncommutative distance between two vertices is always dominated by the geodesic distance, that is, \[d^D(i,j)\le\ell^D(i,j)\] for all $i$, $j$. A survey of general properties can be found in \cite{Besnard}, along with further estimates and comparisons between $d^D$ and $\ell^D$ obtained in relation to structural properties of the graph.

\begin{remark}The natural question of deciding whether the noncommutative distance between two vertices of $G$ increases upon removing an edge currently appears to remain open.
\end{remark}

\subsection{First Reductions}\label{sec-first-reduc}

While the general definition of $d^D(i,j)$ given in \eqref{eq-def-NC-dist} shows the geometric nature of the question of calculating explicitly the noncommutative distance between vertices in a given graph, the problem can be reformulated in somewhat simpler terms. First, a standard C*-algebraic argument (see \cite[Lemma 1]{NC_distances}) allows one to consider only elements $A$ that are positive in $M_n(\C)$ and an equality condition on the norm of commutators $[D,A]$ in \eqref{eq-def-NC-dist}. Therefore, restricting our attention to connected graphs, the distance can be expressed as:
\begin{align*}
 d^D(i,j) &=\max\big\{ \left|a_i-a_j\right|\::\:A = \operatorname{diag}(a_1,\ldots,a_n)\:,\:\big\|[A,D]\big\|=1\:,\:a_1,\ldots,a_n\ge0\big\} \\
          &= \max\big\{ a_i-a_j \::\:A = \operatorname{diag}(a_1,\ldots,a_n)\:,\:\big\|[A,D]\big\|=1\:,\:a_1,\ldots,a_n \in \IR \big\},
\end{align*}
where the second formulation holds because both $\|[A,D]\|$ and $|a_i - a_j|$ are invariant under replacing $A$ by $-A$ or by $A + rI$.

Next, let us notice that if $A = \diag(a_1, \dots, a_n)$ and $D$ satisfies the graph 
compatibility condition \eqref{eq-Dirac-op}, then the commutator $[A,D]$ has $(i,j)$-entry 
$(a_i-a_j) d_{ij}$, and hence is a skew-symmetric matrix. We have the following
canonical form result, e.g., see \cite[Theorem 5]{KL} for 
a simple proof.

\begin{proposition} \label{skew}
Let $T \in M_n$ be a real skew-symmetric matrix. There is an orthogonal 
matrix $Q$ such that $Q^tTQ$ is a direct sum of
$0_{n-2r}$ (the zero matrix in $M_{n-2r}(\R)$) and $2\times 2$
matrices 
$$\begin{pmatrix} 0 & s_j \cr -s_j & 0 \cr\end{pmatrix}, \quad j = 1, \dots, r,$$ 
where $s_1 \ge \cdots \ge s_r > 0$. Consequently,
$T$ has nonzero complex eigenvalues $\pm is_j$ for $j = 1, \dots, r$ and 
$$s_1 = \|T\| = \max\{ x^t T y: x, y \in \IR^n, x^tx = y^t y = 1, x^ty = 0\}.$$
\end{proposition}

By the above proposition, it is clear that
$A = \diag(a_1, \dots, a_n)$ satisfies $\|AD-DA\| \le 1$
if and only if 
\[\tr[A(Dxy^t - xy^tD)] = y^t(AD-DA)x \le 1 \] 
for any orthonormal 
pair of vectors $x,y \in \R^n$.  It follows that:
\begin{equation}\label{eq-NC-dist_trace}
d^D(i,j) = \max\big\{ (a_i - a_j): A  = \diag(a_1, \dots, a_n), \tr(AR)\le 1 \hbox{
for all } R  \in \cS(D)\big\}\end{equation}
where
$$\cS(D)= \{
\{Dxy^t-xy^tD:  x, y \in \IR^n, \|x\| = \|y\| = 1,
x^ty = 0\}.$$

\subsection{Birkhoff-James Orthogonality}\label{sec-BJ}

Orthogonality is a useful tool in the study of Euclidean spaces and the geometric objects modeled after them. However, not all metric spaces of interest come naturally equipped with such a notion. In order to extend techniques that involve perpendicularity (\textit{e.g.} the optimality of orthogonal projections) one may use characterizations expressed in terms of the metric only. One such characterization, introduced by Birkhoff \cite{Birkhoff35} and studied by James \cite{James45,James47}, takes as its starting point the elementary observation that two lines in Euclidean space are perpendicular if and only if no point on the second is closer to a marked point on the first than their intersection.  

\begin{figure}
\begin{center}
\begin{tikzpicture}[scale=1]
\draw[->] (-1.5,1)--(5,1);
\draw[->] (0,-1.2)--(0,5.2);
\draw[thick,dashed] (-0.6,-0.2)--(2,5);
\draw[thick,dotted] (0,1)--(3,2);
\draw (1.5,-1)--(4.2,4.4);
\draw (-2,2)--(3,-0.5);
\draw[very thick,->] (0,1)--(2,0);
\draw[very thick,->] (0,1)--(0.5,2);
\draw[very thick,->] (2,0)--(3,2);
\filldraw (2,0)circle(1.5pt);
\filldraw (3,2)circle(1.5pt);
\draw (0.25,2.3) node {$u$};
\draw (0.8,0.15) node {$v$};
\draw (3.6,1.5) node {$v+tu$};
\draw (-1.4,2.2) node {$\mathcal{L}_1$};
\draw (4.5,3.7) node {$\mathcal{L}_2$};
\end{tikzpicture}
\caption{$\mathcal{L}_1\perp\mathcal{L}_2$ in $\R^2$ Euclidean if and only if $\|v\|\leq \|v+ tu\|$ for all $t$.}\label{fig-BJ}
\end{center}
\end{figure}
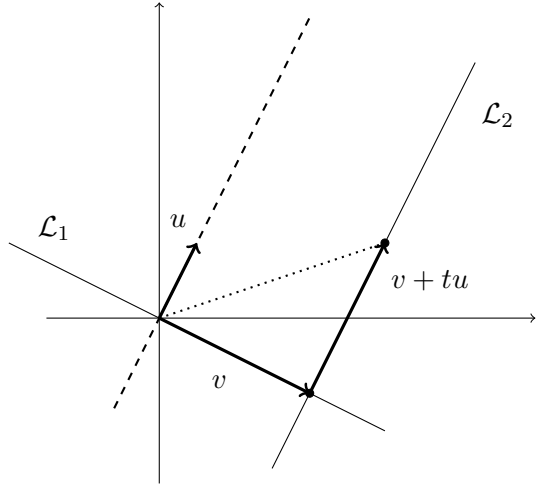

If one denotes the marked point by the origin, the intersection of the lines by $v$, and a direction vector of the second line by $u$ (see Figure \ref{fig-BJ}), then this orthogonality condition can be reformulated as: 
\begin{equation}\label{eq-BJ_ortho}\|v\|\leq \|v+ tu\|\quad\text{for any $t\in\R$}.\end{equation}

An overview of Birkhoff-James orthogonality and its applications in various settings can be found in the survey \cite{BJ_survey}. Its relevance to the study of the problem at hand comes down to the way the definition naturally encapsulates the notion of best approximation to a subspace.

If we focus on a matrix space over ${\mathbb F} = \R$ or $\C$, we say that
a matrix $A \in M_{m,n}(\F)$ is Birkhoff-James orthogonal to a matrix $B \in M_{m,n}(\F)$ 
if 
\begin{equation}\label{E:BJ}
\|A\|\leq \|A+ \gamma B\|\quad\text{for any $\gamma\in\F$}.
\end{equation}
The relation thus obtained (which, we should note, is generally not symmetric) has been studied by Singer \cite{Singer_Best_approx}, as well as by the second-named author and H. Schneider \cite{LS}. The results established in \cite[Theorem 3.1]{LS} characterize Birkhoff-James orthogonality and provide a useful tool for this paper, which can be restated as follows: 

\begin{thm}\label{T:BJ}
Let $A, B \in M_{m,n}(\F)$.  Then $A$ is Birkhoff-James orthogonal to $B$ if and only if there exist unit vectors $x \in \F^m$ and $y \in \F^n$ such that $\|A\| = x^* A y$ and $x^* B y = 0$.
\end{thm}

Finally, we note that Birkhoff-James orthogonality is actually equivalent to a local optimality property, in the sense that \eqref{E:BJ} is equivalent to:
\begin{equation}\label{E:BJ-local}\text{There exists $\epsilon>0$ such that:}\quad\|A\|\leq \|A+ \gamma B\|\quad\text{for any $\gamma\in\F$ with } |\gamma|<\epsilon.
\end{equation}
To see this, suppose $\|A \| \leq \|A + \gamma B\|$.  Then 
$$2\|A\| \leq 2 \|A + \gamma B\| = \|A + (A+2\gamma B)\| \leq \|A\| + \|A + 2\gamma B\|,$$
whence $\|A\| \leq \|A + 2\gamma B\|$.

Thus, in a manner analogous to maximizing a function by finding its critical points, one can use Birkhoff-James orthogonality to find the locations where $d^D(i,j)$ might be achieved.  With the help of Theorem \ref{T:BJ} to provide a tractable characterization, we shall see this idea employed as the key tool in the proof of Proposition \ref{P:irred}.

\section{Path Graphs}\label{sect:path}

Assume $D = (d_{ij}) \in M_n(\C)$  with $d_{i,i+1} = d_{i+1,i} > 0$
for $i = 1, \dots, n-1$, and $d_{i,j} = 0$ otherwise.  It is a general fact (see \cite[Theorem 3]{Besnard}) that the noncommutative distance between any two vertices $i$ and $j$ only depends on the Dirac operator induced by $D$ on the subgraph obtained by taking the union of all paths connecting $i$ to $j$.  As a result, to determine 
\[d^D(i,j) = \max \{a_1 - a_n : A = \diag(a_1, \dots, a_n), \|AD-DA\| = 1, a_1, \dots, a_n \in \IR\},\]
we may focus on the case $(i,j) = (1,n)$, since the union of all paths joining $i$ and $j$ is simply the unique induced path joining $i$ and $j$.

\subsection{Further Reformulation of the Problem for Paths and Viable Vectors}\label{sec-viable_vecs}

We begin by providing some equivalent characterizations of $d^D(1,n)$ that will prove more fruitful to work with.  To simplify notation, we shall write $d_i$ for $d_{i,i+1}$ for the rest of this section; that is, 
$$D = \begin{bmatrix} 0 & d_1 & 0 & \dots & 0 \\ d_1 & 0 & d_2 & & \vdots \\ 0 & d_2 & \ddots & \ddots & 0 \\ \vdots & & \ddots & \ddots & d_{n-1} \\ 0 & \dots & 0 & d_{n-1} & 0 \end{bmatrix}.$$

We will make frequent use of the following matrices related to $D$ throughout the remainder of this paper.

\begin{definition}
 \label{def:LandT}
 Let $d = (d_1, \dots, d_{n-1})$ be fixed, and let $z = (z_1, \dots, z_{n-1}) \in \IR^{n-1}$.  Define $L(d,z)$ (or simply $L(z)$, when $d$ is understood) to be the symmetric tridiagonal matrix with zeros on the diagonal and $d_1 z_1, \dots, d_{n-1} z_{n-1}$ on the superdiagonal, and let $T(d,z)$ (or simply $T(z)$, when $d$ is understood) be the $\lfloor \frac{n+1}{2} \rfloor \times \lfloor \frac{n}{2} \rfloor$ matrix
\begin{equation}\label{E:T(z)}
T(z) = \begin{bmatrix} d_1 z_1 & & \\ d_2 z_2 & d_3 z_3 & \\ & d_4 z_4 & d_5 z_5 & \\ & & & \ddots \end{bmatrix}.
\end{equation} 
\end{definition}

\begin{remark}
\label{rem:dtomunu}
We remark that $d_iz_i$ will be in row $\left\lceil \frac{i+1}{2}\right\rceil = \left\lfloor \frac{i + 2}{2} \right\rfloor$ and column $\left\lceil \frac{i}{2}\right\rceil = \left\lfloor \frac{i + 1}{2} \right\rfloor$ of $T(z)$.
\end{remark}

\begin{lemma}\label{L:equiv}
Suppose $D \in M_n(\IR)$ is a symmetric tridiagonal matrix with zeros on the diagonal and $d_1, \dots, d_{n-1} > 0$ on the superdiagonal. Then, 
\begin{align*}
d^D(1,n) = &\max \{a_1 - a_n : A = \diag(a_1, \dots, a_n), \|AD-DA\| = 1, a_1, \dots, a_n \in \IR\} \\
= &\max \{z_1 + \dots + z_{n-1} : z_j \geq 0, \; \rho(L(z)) = 1\} \\
= &\max \{z_1 + \dots + z_{n-1} : z_j \geq 0, \; \|T(z)\| = 1\}.
\end{align*}
\end{lemma}

\begin{proof}
By compactness, all maxima exist.  Writing $z_j = a_j - a_{j+1}$ for $j= 1, \dots, n-1$, we have $AD - DA = \sum_{j=1}^{n-1} z_j d_j (E_{j,j+1} - E_{j+1,j})$.  
Let $Q = \diag(i, i^2, \dots, i^n)$. Then, 
\[i Q^*(AD-DA)Q = \sum_{j=1}^{n-1} z_j d_j (E_{j,j+1}+E_{j+1,j}) = L(z).\] 
Since $\|iQ^*(AD-DA)Q\| = \|AD-DA\|$, the first maximum equals
$$\max \{z_1 + \dots + z_{n-1} : z \in \IR^{n-1}, \; \|L(z)\| \leq 1\}.$$
Furthermore, there exists a matrix $G = \diag(g_1, \dots, g_n)$, with $g_1, \dots, g_n \in \{1, -1\}$, such that
\[GL(z)G 
= \sum_{j=1}^{n-1} g_j g_{j+1} z_j d_j (E_{j,j+1}+E_{j+1,j})
= \sum_{j=1}^{n-1} |z_j| d_j (E_{j,j+1}+E_{j+1,j}).\]
Since $G$ is orthogonal, $\|GL(z)G\| = \|L(z)\|$, so we may assume $z_j \geq 0$ for all $j$.  Since $L(z)$ is a real symmetric matrix, its norm equals its spectral radius $\rho(L(z))$, so the first two maxima coincide.

Note that $T(z)$ is the submatrix of $L(z)$ with $1,3,5, \dots$ as row indices and $2, 4, 6, \dots$ as column indices, so $L(z)$ is permutationally similar  to 
$$\begin{bmatrix} 0_k & T(z) \\ T(z)^t & 0_{n-k} \end{bmatrix}$$
where $k = {\lfloor \frac{n+1}{2} \rfloor}$.  Thus the eigenvalues of $L(z)$ occur in plus/minus pairs, and the nonnegative eigenvalues of $L(z)$ are precisely the singular values of $T(z)$, whence the second and third maxima coincide.
\end{proof}

There are now two situations to consider, depending on whether or not $L(z)$ and $T(z)$ in Lemma \ref{L:equiv} are both reducible or irreducible.  Observe that $L(z)$ is irreducible precisely when $z_1, \dots, z_{n-1}$ are all nonzero.  We consider this irreducible case first; to concisely state the result, we need to introduce some notation.

\begin{definition}
\label{def:munu}
Given a vector of positive numbers $d=(d_1, \dots, d_{n-1})$, set $m = \lfloor \frac{n+1}{2} \rfloor$.  Define vectors $\mu(d) \in \IR^m$ and $\nu(d) \in \IR^{n-m}$ (or, simply $\mu$ and $\nu$ when $d$ is understood) by defining their entries to be 
\begin{equation}
	\mu_1 = \nu_1 = 1, \qquad \mu_{k+1} := \prod_{j = 1}^{k} \frac{d_{2j-1}}{d_{2j}},\qquad \nu_{k+1} := \prod_{j = 1}^{k} \frac{d_{2j}}{d_{2j+1}}.
\end{equation} 
\end{definition}

\begin{lem}
\label{lem:muTnu}
Let $\mu$ and $\nu$ be as in Definition \ref{def:munu}.  Then,
\[ \mu^t T(z) \nu = d_1 \cdot \sum_{i = 1}^{n-1} z_i.\]
\end{lem}

\begin{proof}
First, we remark that
\[ \mu^t T(z) = (d_1(z_1 + z_2), \frac{d_1d_3}{d_2}(z_3 + z_4), \dots) = (\frac{d_1}{\nu_1}(z_1 + z_2), \frac{d_1}{\nu_2}(z_3 + z_4), \dots),\]
where the last term of $\mu^t T(z)$ is either $\frac{d_1}{\nu_{\lfloor n/2 \rfloor}}(z_{n-2} + z_{n-1})$ or $\frac{d_1}{\nu_{\lfloor n/2 \rfloor}}z_{n-1}$, depending on whether $n$ is odd or even.  Thus,
\[ \mu^t T(z) \nu = d_1 \sum_{i = 1}^{n-1} z_i,\]
as desired.
\end{proof}

\begin{proposition}\label{P:irred}
Suppose that the maxima in Lemma \ref{L:equiv} are attained at $z = (z_1, \dots, z_{n-1}) \in \IR^{n-1}$ with positive entries.  Then the maximum value of $\sum_{j=1}^{n-1} z_j$ 
is given by 
\begin{equation}\label{E:max-irred}
\sum_{j=1}^{n-1} z_j = \frac{\|\mu\| \|\nu\|}{d_1} = \frac{1}{d_1} \sqrt{1 + \mu_2^2 + \dots + \mu_{\lfloor \frac{n+1}{2} \rfloor}^2} \sqrt{1 + \nu_2^2 + \dots + \nu_{\lfloor \frac{n}{2} \rfloor}^2}.
\end{equation}
\end{proposition}

\begin{proof}
Using the notation in Lemma \ref{L:equiv}, since $z_j > 0$ for all $j$, $L(z)$ is an irreducible nonnegative matrix, so by the Perron-Frobenius theorem there exists a unique unit vector $v \in \IR^n$ with positive entries such that $L(z) v = \rho(L(z)) v = v$; moreover, 1 is a simple eigenvalue.  Since the nonnegative eigenvalues of $L(z)$ are the singular values of $T(z)$, $T(z)$ has a unique norm-attaining vector (modulo scalar multiplication).

Let $m = {\lfloor \frac{n+1}{2} \rfloor}$. 
Because $$\begin{bmatrix} 0_m & T(z) \\ T(z)^t & 0_{n-m} \end{bmatrix}$$ 
is permutationally similar to $L(z)$, there exist positive vectors $x \in \IR^m$ and $y \in \IR^{n-m}$ such that 
$$\begin{bmatrix} 0_m & T(z) \\ T(z)^t & 0_{n-m} \end{bmatrix} \begin{bmatrix} x \\ y \end{bmatrix} = \begin{bmatrix} x \\ y \end{bmatrix},$$
so $T(z)y=x$ and $T(z)^t x = y$.  Then $\|x\|^2 = \|T(z)y\|^2 = y^t T(z)^t T(z) y = y^t y = \|y\|^2$, so without loss of generality we may assume that both $x$ and $y$ are unit vectors.  Then $x, y$ are the unique positive unit vectors satisfying $x^t T(z) y = 1$.

Note $T(z) = \sum_{j=1}^{n-1} z_j R_j$ where $R_{2k} = d_{2k} E_{k+1,k}$ and $R_{2k-1} = d_{2k-1} E_{k,k}$.  We claim that $\|T(z)+t(R_j -R_{j+1})\| \ge \|T(z)\|$ for all $t\in \IR$ satisfying $|t| < \min\{z_j, z_{j+1}\}$.  By way of contradiction, suppose there exists $t \in \IR$ with $|t| <  \min\{z_j, z_{j+1}\}$ such that $\|T(z)+t(R_j -R_{j+1})\| = d < 1$.
If 
$$(\tilde z_1, \dots, \tilde z_n)  = (z_1, \dots, z_{j-1}, z_j+t, z_{j+1}-t, z_{j+2}, \dots, z_n)/d,$$
then 
$$\|T(\tilde{z})\| = \|T(z)+t(R_j -R_{j+1})\|/d =1$$
and 
$$\sum_{j=1}^{n-1}\tilde z_j = \sum_{j=1}^{n-1} z_j/d > \sum_{j-1}^{n-1} z_j,$$
which contradicts that $z$ gives a maximum, proving the claim.

By \eqref{E:BJ-local}, $T(z)$ is Birkhoff-James orthogonal to $R_j -R_{j+1}$ for each $j$.  Then by Theorem \ref{T:BJ} there exist unit vectors $u, v$ such that $u^t T(z)v = 1$ and $u^t (R_j - R_{j+1}) v = 0$.  By the aforementioned uniqueness we may take $u=x$ and $v=y$.  Thus $x^t R_j y = x^t R_{j+1} y$ for $j=1, \dots, n-1$, which implies
\begin{equation}\label{E:recur-xy}
x_{k+1} = \frac{d_{2k-1}}{d_{2k}} x_k \quad \text{ and } \quad y_{k+1} = \frac{d_{2k}}{d_{2k+1}} y_k
\end{equation}
for $k \leq (n-1)/2$ and $k \leq (n-2)/2$ respectively.

As in Definition \ref{def:munu}, let $\mu$ and $\nu$ be vectors whose entries $\mu_k$ and $\nu_k$ for $k\geq1$ are
\begin{equation}\label{eq-def_muk_nuk}
	\mu_1 = \nu_1 = 1, \qquad \mu_{k+1} := \prod_{j = 1}^{k} \frac{d_{2j-1}}{d_{2j}},\qquad \nu_{k+1} := \prod_{j = 1}^{k} \frac{d_{2j}}{d_{2j+1}}.
\end{equation} 
By the recurrence relations \eqref{E:recur-xy}, and since $x,y$ are positive unit vectors, $x = \mu/\|\mu\|$, $y = \nu/\|\nu\|$.

Finally, using Lemma \ref{lem:muTnu}, we have 
\[1 = x^t T(z) y = \frac{\mu^t T(z) \nu}{\|\mu\| \|\nu\|} = \frac{d_1}{\|\mu\| \|\nu\|} \sum_{i = 1}^{n-1} z_i,\]
so 
\begin{equation}
\sum_{j=1}^{n-1} z_j = \frac{1}{d_1} \sqrt{1 + \mu_2^2 + \dots + \mu_{\lfloor \frac{n+1}{2} \rfloor}^2} \sqrt{1 + \nu_2^2 + \dots + \nu_{\lfloor \frac{n}{2} \rfloor}^2}.
\end{equation}
\end{proof}

In general the maxima in Lemma \ref{L:equiv} may occur at a nonnegative vector $z \in \IR^{n-1}$ with some zero entries.  In such a case $T(z)$ may be reducible, and the decomposition occurs at the locations where $z_j = 0$.  We introduce further notation.

\begin{definition}
\label{def:viablesetup}
Let $z \in \IR^{n-1}$ be a nonnegative vector.  Recursively define integers $\alpha_k$, $\beta_k$ for $k \geq 1$ as follows:
\begin{itemize}
\item $\beta_0 = 0$,
\item $\alpha_k$ is the least integer $i > \beta_{k-1}$ such that $z_i \ne 0$, and
\item $\beta_k$ is the greatest integer $i$ such that $z_{\alpha_k}, \dots, z_i$ are all nonzero.  (Note that $\beta_k = \alpha_k$ is possible.)
\end{itemize}
Further define 
\[ f_k \colonequals (d_{\alpha_k}, \dots, d_{\beta_k}),\quad\quad u_k \colonequals (z_{\alpha_k}, \dots, z_{\beta_k}),\quad\text{and}\quad T_k(z) \colonequals T(f_k, u_k).\]
\end{definition}

In general we do not have $T(z) = T_1(z) \oplus \dots \oplus T_r(z)$, but we do have $T(z) = \Theta_1 \oplus \dots \oplus \Theta_s$, where each $\Theta_j$ is either a zero matrix of some dimension, or is some $T_i(z)$ or $T_i(z)^t$, possibly padded with some zero rows or columns.  We illustrate with an example.  Suppose
$$T(z) = \begin{bmatrix}
0 & &  & & & \\
b & c & & &  \\
 & d & e & & & \\
 &&0&0& & & \\
 &&&h&i&&\\
 &&&&j&k&  \\
 &&&&&l&0
\end{bmatrix}.$$
Then, $$T_1(z) = \begin{bmatrix} b &  \\ c & d  \\ & e  \end{bmatrix} \qquad \text{ and } T_2(z) = \begin{bmatrix} h & \\ i & j \\ & k &l \end{bmatrix},$$
and one could choose $\Theta_1 = \begin{bmatrix} 0 \\ T_1^t(z) \\ 0 \end{bmatrix}$ and $\Theta_2 = \begin{bmatrix} T_2^t(z) & 0 \end{bmatrix}$ (other possibilities exist).  Despite the non-uniqueness of this decomposition, it is however the case that
$$\|T(z)\| = \|\Theta_1 \oplus \dots \oplus \Theta_s\|
= \max_k \|T_k(z)\|.$$

Then, Lemma \ref{L:equiv} gives
\begin{align*}
    d^D(1,n) &= \max \{z_1 + \dots + z_{n-1} : z_j \geq 0 , \|T_k(z)\| \leq 1\} \\
    &= \max \left\{ \sum_{k=1}^r \sum_{j=\alpha_k}^{\beta_k} z_j : z_j \geq 0, \|T_k(z)\| \leq 1 \right\} \\
    &= \sum_{k=1}^r  \max \left\{ \sum_{j=\alpha_k}^{\beta_k} z_j : z_j \geq 0, \|T_k(z)\| \leq 1 \right\}.
\end{align*}
Thus the problem reduces to applying Proposition \ref{P:irred} to $T_1(z), \dots, T_r(z)$, and the maximum value of $\sum_{j=1}^{n-1} z_j$ is found by summing the corresponding maxima for $T_1(z), \dots, T_r(z)$.  

Moreover, it turns out that, to achieve the maximum, there cannot be any extra zeros in the $\Theta_j$ summands of $T(z)$.  To summarize all of these results, we first introduce the following definition.

\begin{definition}\label{D:viable} 
We say that a nonnegative vector $z \in \IR^{n-1}$ is \textit{viable} if the following hold. 
\begin{enumerate}
    \item[(1)] There are no consecutive zeros in $z$, and $z_1 \ne 0$, $z_{n-1} \ne 0$.  (This means that $T(z)$ actually equals the direct sum of $T_1(z), \dots, T_r(z)$, or possibly some of their transposes.)
    \item[(2)] For each $k=1, \dots, r$, \[T_k(z) y_k = x_k\qquad \text{and}\qquad T_k(z)^t x_k = y_k,\] where $x_k = \dfrac{\mu(f_k)}{\|\mu(f_k)\|}$ and  $y_k = \dfrac{\nu(f_k)}{\|\nu(f_k)\|}$.
\end{enumerate}
\end{definition}

\begin{remark}
\label{rem:viable}
\begin{enumerate}
    \item[(1)] Note that it is possible that there is no positive vector $z \in \IR^{n-1}$ that is viable, since it is possible that no positive vector satisfies the second condition of Definition \ref{D:viable}. However, it is not difficult to see that a viable vector always exists: for example, if $n$ is even, then
    \[ z = \left(\frac{1}{d_1}, 0, \frac{1}{d_3}, 0, \frac{1}{d_5}, 
    \dots, 0, \frac{1}{d_{n-1}}\right)\]
    is viable, and, if $n$ is odd, then
    \[ z = \left(\frac{d_2}{d_1\sqrt{d_1^2 + d_2^2}}, \frac{d_1}{d_2\sqrt{d_1^2 + d_2^2}}, 0, \frac{1}{d_4}, 0, \frac{1}{d_6}, 0, \dots, 0, \frac{1}{d_{n-1}} \right)\]
    is viable.
    \item[(2)] If $z$ is viable and $T(z) = \Theta_1 \oplus \cdots \oplus \Theta_r$, then $\Theta_k = T_k(z)$ if $\alpha_k$ is odd (since there is a unique nonzero element in the row of $T(z)$ containing $d_{\alpha_k}z_{\alpha_k}$), whereas, if $\alpha_k$ is even, then $\Theta_k = T_k(z)^t$ (since either the row of $T(z)$ containing $d_{\alpha_k}z_{\alpha_k}$ contains two nonzero entries or $T_k(z) \in M_1(\IR)$). 
\end{enumerate}
\end{remark}

\begin{thm}
\label{thm:maxatviabledecomp}
Let $D \in M_n(\IR)$ be a symmetric tridiagonal matrix with zeros on the diagonal and $d_1, \dots, d_{n-1} > 0$ on the superdiagonal.

Then, 
\[ d^D(1,n) = \max \left\{ \sum_k \frac{\|\mu(f_k) \| \|\nu(f_k) \|}{d_{\alpha_k}} : z \in \IR^{n-1} \emph{ is viable }\right\}.\]
\end{thm}

\begin{proof}
By Lemma \ref{L:equiv}, Proposition \ref{P:irred}, and the preceding discussion, $d^D(1,n)$ is achieved at some decomposition with the asserted formula.  There remain two things to check: (i) first, we must show that none of the summands 
in a direct sum decomposition of $T(z)$ has a zero row or column (so that $z$ is indeed viable), 
and (ii) second, we must ensure that all viable $z$ are legitimate candidates, that is, we must ensure that the decomposition of $T(z)$ for a viable $z$ satisfies the Birkhoff-James orthogonality conditions and the constraints $\|T_k(z)\|=1$.

To prove assertion (i), it will be easier to work with $L(z)$ as defined in Lemma \ref{L:equiv}.  Because $L(z)$ is permutationally similar to 
$$\begin{bmatrix} 0 & T(z) \\ T(z)^t & 0 \end{bmatrix},$$
the direct sum decomposition of $L(z)$ will have a $1 \times 1$ zero block if and only if some $\Theta_j$ in the decomposition $T(z) = \Theta_1 \oplus \dots \oplus \Theta_s$ has a zero row or column (i.e. $z$ is not viable).

We shall show that if the decomposition of $L(z)$ has a $1 \times 1$ zero block, then we can always combine it with another block to increase the objective function $z_1 + \dots + z_{n-1}$, so $d^D(1,n)$ cannot be attained at such a non-viable $z$.

Without loss of generality, up to a permutation similarity, we may assume all zero summands of $L(z)$ occur last.  To this end, suppose $L_{k+1}(z)$ is the first zero summand and $L_k(z)$ is an irreducible (nonzero) summand.  To simplify notation, without loss of generality we may suppose $k=1$ and $L_1(z)$ is an $m \times m$ matrix.  Then $z_1, \dots, z_{m-1} >0, z_m = 0$.

Suppose, to the contrary, that $d^D(1,n)$ is attained at this decomposition.  By Lemma \ref{L:equiv}, $$z_1 + \dots + z_{m} = \max\{w_1 + \dots + w_m : w_j \geq 0, \, \rho(M(w)) = 1 \},$$
where $M(w) \in M_{m+1}$ is the symmetric tridiagonal matrix 
$M(w) = \sum_{j=1}^m d_j w_j (E_{j,j+1} + E_{j+1,j})$.
Define $z_m' = 2(1-t)(z_1 + \dots + z_{m-1})$ and $z_j' = t z_j$ for $j=1, \dots, m-1$, where  we choose $t \in (0,1)$ 
to satisfy 
\begin{equation}\label{E:t}
4(1-t) d_m^2 \sum_{j=1}^{m-1} z_j^2 < 1.
\end{equation}

Then $\sum_{j=1}^m z_j' = (2-t) \sum_{j=1}^m z_j > \sum_{j=1}^m z_j$, so by our assumption of maximality, $\rho(M(z')) > 1$.  We shall show that 
in fact 
$\rho(M(z')) \leq 1$, giving the desired contradiction.

Because $M(z')$ is a real symmetric matrix whose eigenvalues occur in plus/minus pairs, the condition $\rho(M(z')) \leq 1$ is equivalent to $I + M(z')$ being positive semidefinite.
Writing $B = \sum_{j=1}^{m-1} z_j d_j (E_{j,j+1}+E_{j+1,j})$ and $v = (0, \dots, 0, z_{m}' d_{m})^t$, we have
$$I+M(z') = \begin{bmatrix} I_{m} + t B & v \cr v^t & 1 \cr\end{bmatrix}.$$
By our assumption of maximality, $\rho(B) = 1$, so $I_m + tB$ is positive definite, and hence invertible.
Let $S = \begin{bmatrix} I_{m} & (I_{m} + tB)^{-1} v \\ 0 & -1 \end{bmatrix}$, so $S$ is an invertible matrix satisfying 
\begin{equation}\label{E:congruence}
	S^t (I + M(z')) S = (I_{m}+tB) \oplus [1 - (z'_{m})^2 d_{m}^2 \xi],
\end{equation}
where $\xi$ is the $(m,m)$ entry of $(I_m+tB)^{-1}$.  By writing $B = W^t \diag(\mu_1, \dots, \mu_{m})W$ for an orthogonal matrix $W = (w_{ij}) \in M_{m}$, we have
$$\xi = (W^t(I+t\diag(\mu_1, \dots, \mu_{m})^{-1}W)_{m,m} = \sum_{j=1} \frac{w_{j,m}^2}{1+t\mu_j};$$
note that $|\mu_j| \leq 1$ because $\rho(B) = 1$.

Now $\rho(M(z')) \leq 1$ is equivalent to $I+M(z')$ being positive semidefinite; by the congruence \eqref{E:congruence}, this is equivalent to $(z'_m)^2 d_m^2 \xi \leq 1$.  Indeed,
\begin{align*}
	(z'_m)^2 d_m^2 \xi &= 4(1-t)^2 \left( \sum_{j=1}^{m-1} z_j^2 \right) d_m^2 \sum_{j=1} \frac{w_{j,m}^2}{1+t\mu_j} < (1-t)  \sum_{j=1} \frac{w_{j,m}^2}{1+t\mu_j} \qquad \text{ by } \eqref{E:t} \\
	&\leq (1-t) \sum_{j=1} \frac{w_{j,m}^2}{1-t} = 1 
\end{align*}
because $W$ is orthogonal, and we obtain the desired contradiction, proving (i).

To prove assertion (ii), recall that, by Definition \ref{D:viable}, $T_k(z) y_k = x_k$ and $T_k(z)^t x_k = y_k$ for a viable $z$.  Thus
$$\begin{bmatrix}0 & T_k(z) \\ T_k(z)^t & 0 \end{bmatrix} \begin{bmatrix} x_k \\ y_k \end{bmatrix} = \begin{bmatrix} x_k \\ y_k \end{bmatrix};$$
by the Perron-Frobenius theorem, the spectral radius of the block matrix is 1 (note that the block matrix is nonnegative and irreducible because $T_k(z)$ is).  Since the eigenvalues of the block matrix are plus/minus pairs of the singular values of $T_k(z)$, we have $\|T_k(z)\| = 1$.  The Birkhoff-James orthogonality conditions in the proof of Proposition \ref{P:irred} then follow from the way $x_k$ and $y_k$ are defined.
\end{proof}

\subsection{Determination of Viable Vectors}\label{sec-viable_det}

In general, to compute $d^D(1,n)$ one must find and check all viable $z \in \IR^{n-1}$.  Toward this end, we present some results that narrow the search considerably.

\begin{lem}
 \label{lem:prerefine}
 Let $1 < p < n-1$.  If 
 \[ f_1 \colonequals (d_1, \dots, d_{p-1}), \;
 f_2 \colonequals (d_{p+1}, \dots, d_{n-1}),\]
 \[ \fraka \colonequals \frac{\|\mu\| \|\nu\|}{d_1},\; \frakb \colonequals \frac{\|\mu(f_1)\| \|\nu(f_1)\|}{d_1}, \; \frakc \colonequals \frac{\|\mu(f_2)\| \|\nu(f_2)\|}{d_{p+1}},\]
 then
 \[ \fraka \ge \frakb + \frakc.\]
\end{lem}

\begin{proof}
 Let $\mu_r$ be the first $\mu_i$ involving $d_{p+1}$ and let $\nu_s$ be the first $\nu_i$ involving $d_{p+1}$. Thus $\mu_r, \nu_s$ are precisely the products of terms up to, but not including, 
$d_{p-1}/d_{p}$ and $d_{p}/d_{p+1}$.  Note $r, s$ may equal 1.
Thus,
$$\frakb = \frac{1}{d_{1}} 
\sqrt{1 + \mu_1^2 + \mu_2^2 + \cdots + \mu_{r-1}^2}
\sqrt{1 + \nu_1^2 + \nu_2^2 + \cdots + \nu_{s-1}^2},$$
while
\begin{align*}
    \frakc &= \frac{1}{d_{p+1}} \sqrt{1 + (\mu_{r+1}/\mu_r)^2 + \dots 
    + (\mu_{\lfloor (n+1)/2 \rfloor} /\mu_r)^2} \sqrt{1 + (\nu_{s+1}/\nu_s)^2 + \dots + (\nu_{\lfloor n/2 \rfloor}/\nu_s)^2} \\
    &= \frac{1}{d_{p+1} \mu_r \nu_s} \sqrt{\mu_r^2 + 
    \dots +  \mu_{\lfloor (n+1)/2 \rfloor}^2} \sqrt{\nu_s^2 + \dots +  \nu_{\lfloor n/2 \rfloor}^2} \\
    &= \frac{1}{d_{1}} \sqrt{\mu_r^2 + \dots +  \mu_{\lfloor (n+1)/2 \rfloor}^2} \sqrt{\nu_s^2 + \dots +  \nu_{\lfloor n/2 \rfloor}^2}.
\end{align*}
Let 
$$u_1 = 1 + \mu_1^2 + \cdots + \mu_{r-1}^2, \qquad 
u_2 = \mu_r^2 +  \cdots + \mu_{\lfloor (n+1)/2 \rfloor}^2,$$
$$v_1 = 1 + \nu_1^2 + \cdots + \nu_{s-1}^2, \qquad
v_2 = \nu_s^2 +  \cdots +  \nu_{\lfloor n/2 \rfloor}^2.$$
Then 
$$\fraka^2 - (\frakb+\frakc)^2
= \frac{1}{d_{12}^2}\left(\sqrt{u_1 v_2} - \sqrt{v_1 u_2}\right)^2 \ge 0,$$
as desired.
\end{proof}

We can use Lemma \ref{lem:prerefine} to prove the following, which shows that, once a viable $z$ is found, one need not check refinements of the decomposition.

\begin{prop}
 \label{prop:refine}
 Let $D \in M_n(\IR)$ be a symmetric tridiagonal matrix with zeros on the diagonal and $d_1, \dots, d_{n-1} > 0$ on the superdiagonal.  If $z$ and $z'$ are both viable and $z_i' = 0$ whenever $z_i = 0$, then 
  \[ \sum_{i = 1}^{n-1} z_i \ge \sum_{i = 1}^{n-1} z_i'.\]
\end{prop}

\begin{proof}
By induction, it suffices to consider the result when all $z_i, z_i'$ are positive except $z_p' = 0$ for some unique $1 \le p \le n-1$. Note that if $p=1$ or $p = n-1$, then $z'$ is not viable, a contradiction.
 
Thus, we may assume $1 < p < n-1$.  By Proposition \ref{P:irred},
\[ \sum_{i = 1}^{n-1} z_i = \frac{\|\mu\| \|\nu\|}{d_1},\]
and if 
\[ f_1 \colonequals (d_1, \dots, d_{p-1}), \;
 f_2 \colonequals (d_{p+1}, \dots, d_{n-1}),\]
then, by Theorem \ref{thm:maxatviabledecomp}, 
\[
\sum_{i = 1}^{n-1} z_i' = \left(\sum_{i = 1}^{p-1} z_i'\right) + \left(\sum_{i=p+1}^{n-1} z_i' \right) =  \frac{\|\mu(f_1)\| \|\nu(f_1)\|}{d_1} + \frac{\|\mu(f_2)\| \|\nu(f_2)\|}{d_{p+1}},
\]
and the result now follows by Lemma \ref{lem:prerefine}.
\end{proof}

\begin{remark}
 It is possible for $z$ and $z'$ both to be viable, $z'$ to have more zero entries than $z$, but 
 \[ \sum_{i = 1}^{n-1} z_i \le \sum_{i = 1}^{n-1} z_i';\]
 Proposition \ref{prop:refine} applies only when the decomposition of $T(z')$ is a refinement of the decomposition of $T(z)$.  For example, if $n = 6$ and
 $d = (1, 2, 1, 2, 1),$ then one can check that there does not exist a positive vector that is viable (see Proposition \ref{prop:allviablez} below), but both \[ z = \left(\frac{2}{\sqrt{5}}, \frac{1}{2\sqrt{5}}, 0, \frac{1}{2\sqrt{5}}, \frac{2}{\sqrt{5}} \right)\]
 and $z' = (1,0,1,0,1)$ are viable.  In this case,
 \[ \sum_{i = 1}^5 z_i = \sqrt{5} < 3 = \sum_{i = 1}^5 z_i',\]
 even though $z'$ has more zero entries than $z$.  (For more information about verifying the claims in this remark, see Example \ref{ex:6}.)
\end{remark}

Our goal now is to provide a method for determining all viable $z$.  We first need a definition.

\begin{definition}
\label{def:truncmunu}
Let $\mu$ and $\nu$ be as in Definition \ref{def:munu}, and let $1 \le a \le b \le n-1$.  Define 
\[ \mu[a,b] \colonequals (\mu_i, \dots, \mu_j)^t,\]
where $d_az_a$ is in row $i$ of $T(z)$ and $d_bz_b$ is in row $j$ of $T(z)$, and define
\[ \nu[a,b] \colonequals (\nu_k, \dots, \nu_\ell)^t,\]
where $d_az_a$ is in column $k$ of $T(z)$ and $d_bz_b$ is in column $\ell$ of $T(z)$.  Note that, by Remark \ref{rem:dtomunu}, we have
\[ \mu[a,b] = \left(\mu_{\lceil\frac{a + 1}{2}\rceil}, \dots, \mu_{\lceil\frac{b+1}{2}\rceil} \right)^t\qquad \hbox{ and } \qquad
\nu[a,b] = \left( \nu_{\lceil\frac{a}{2}\rceil}, \dots, \nu_{\lceil\frac{b}{2}\rceil}\right)^t.\]
\end{definition}

The following result provides our method for finding all viable $z$, which is not difficult to implement in practice.

\begin{prop}
 \label{prop:allviablez}
 Let $\mu$ and $\nu$ be as in Definition \ref{def:munu}.  If $z \in \IR^{n-1}$ is viable, then \[T(z)T(z)^t \mu = \mu\qquad \text{and} \qquad T(z)^t T(z) \nu = \nu.\]
\end{prop}

Before proving Proposition \ref{prop:allviablez}, we present the following example, which illustrates the main ideas of the proof.

\begin{example}
Let $n = 10$.  We begin with a viable vector $z$.  The coordinates in which $z$ is $0$ determines the decomposition of $T(z)$ into $\Theta_k$.  Assume our viable vector (which has $n - 1 = 9$ entries) is

\[ z = (z_1, z_2, 0, z_4, z_5, 0, z_7, z_8, z_9).\]

Since there are two zeros, we have $T(z) = \Theta_1 \oplus \Theta_2 \oplus \Theta_3$.  Indeed, 

\[ T(z) = \begin{bmatrix}
           d_1z_1 & 0      & 0      & 0      & 0\\
           d_2z_2 & 0      & 0      & 0      & 0\\
           0      & d_4z_4 & d_5z_5 & 0      & 0\\
           0      & 0      & 0      & d_7z_7 & 0\\
           0      & 0      & 0      & d_8z_8 & d_9z_9\\
          \end{bmatrix},
\]
 
and we have $\Theta_1 = T_1(z)$, where
\[ T_1(z) = \begin{bmatrix}
             d_1z_1 \\
             d_2z_2
            \end{bmatrix},
\]
$\Theta_2 = T_2(z)^t$, where
\[ T_2(z) = \begin{bmatrix}
             d_4z_4 \\
             d_5z_5
            \end{bmatrix},
\]
$\Theta_3 = T_3(z)$, where 
\[ T_3(z) = \begin{bmatrix}
             d_7z_7 & 0 \\
             d_8z_8 & d_9z_9
            \end{bmatrix}.
\]
Note that, by Definition \ref{def:viablesetup}, we have $\alpha_1 = 1$, $\beta_1 = 2$, $\alpha_2 = 4$, $\beta_2 = 5$, $\alpha_3 = 7$, $\beta_3 = 9$, and so $f_1 = (d_1, d_2)$, $f_2 = (d_4, d_5)$, $f_3 = (d_7, d_8, d_9)$.
In each case, $x_k = \mu(f_k)/\|\mu(f_k)\|$, $y_k = \nu(f_k)/\|\nu(f_k)\|$, and, by definition of viable, $T_k(z)y_k = x_k$, $T_k(z)^t x_k = y_k$;  that is, $T_k(z)T_k(z)^t x_k = x_k$, $T_k(z)^tT_k(z) y_k = y_k$, and hence $T_k(z)T_k(z)^t \mu(f_k) = \mu(f_k)$, $T_k(z)^tT_k(z) \nu(f_k) = \nu(f_k)$.

Now, by definition, we have 
\[ \mu(f_1) = \left(1, \frac{d_1}{d_2}\right)^t, \quad
\mu(f_2) = \left(1, \frac{d_4}{d_5}\right)^t, \quad  
\mu(f_3) = \left(1, \frac{d_7}{d_8}\right)^t,
\]
\[ \nu(f_1) = (1), \qquad \nu(f_2) = (1), \qquad  
\nu(f_3) = \left(1, \frac{d_8}{d_9}\right)^t. 
\]

Recall from Definition \ref{def:truncmunu} that 
\[\mu[\alpha_k, \beta_k] \colonequals \left(\mu_{\left\lceil\frac{\alpha_k + 1}{2}\right\rceil}, \dots, \mu_{\left\lceil\frac{\beta_k + 1}{2}\right\rceil}\right)^t,\] 
\[\nu[\alpha_k, \beta_k] \colonequals \left(\nu_{\left\lceil \frac{\alpha_k}{2} \right\rceil}, \dots, \nu_{\left\lceil \frac{\beta_k}{2} \right\rceil} \right)^t.\]

Noting that 
\[ \mu = (\mu_1, \mu_2, \mu_3, \mu_4, \mu_5)^t = \left(1, \frac{d_1}{d_2}, \frac{d_1d_3}{d_2d_4}, \frac{d_1d_3d_5}{d_2d_4d_6}, \frac{d_1d_3d_5d_7}{d_2d_4d_6d_8} \right)^t,\]
\[ \nu = (\nu_1, \nu_2, \nu_3, \nu_4, \nu_5)^t = \left(1, \frac{d_2}{d_3}, \frac{d_2d_4}{d_3d_5}, \frac{d_2d_4d_6}{d_3d_5d_7}, \frac{d_2d_4d_6d_8}{d_3d_5d_7d_9} \right)^t,\]
we have
\[ \mu[\alpha_1, \beta_1] = (\mu_1, \mu_2)^t = \mu_1 \cdot \mu(f_1),\]
\[ \mu[\alpha_2, \beta_2] = (\mu_3) = \mu_3 \cdot \nu(f_2),\]
\[ \mu[\alpha_3, \beta_3] = (\mu_4, \mu_5)^t = \mu_4 \cdot \mu(f_3),\]
and further
\[ \nu[\alpha_1, \beta_1] = (\nu_1) = \nu_1 \cdot \nu(f_1),\]
\[ \nu[\alpha_2, \beta_2] = (\nu_2, \nu_3)^t = \nu_2 \cdot \mu(f_2),\]
\[ \nu[\alpha_3, \beta_3] = (\nu_4, \nu_5)^t = \nu_4 \cdot \nu(f_3).\]

This shows that, in each case,
\[ \Theta_k \Theta_k^t \mu[\alpha_k, \beta_k] = \mu[\alpha_k, \beta_k],\]
\[ \Theta_k^t \Theta_k \nu[\alpha_k, \beta_k] = \nu[\alpha_k, \beta_k],\]
and hence 
\[ T(z)T(z)^t \mu = \begin{bmatrix}
\Theta_1 \Theta_1^t \mu[\alpha_1, \beta_1] \\
\Theta_2 \Theta_2^t \mu[\alpha_2, \beta_2] \\
\Theta_3 \Theta_3^t \mu[\alpha_3, \beta_3] \\
\end{bmatrix}
= \begin{bmatrix}
\mu[\alpha_1, \beta_1] \\
\mu[\alpha_2, \beta_2] \\
\mu[\alpha_3, \beta_3] \\
\end{bmatrix}
= \mu,\] 
\[ T(z)^tT(z) \nu = \begin{bmatrix}
\Theta_1^t \Theta_1 \nu[\alpha_1, \beta_1] \\
\Theta_2^t \Theta_2 \nu[\alpha_2, \beta_2] \\
\Theta_3^t \Theta_3 \nu[\alpha_3, \beta_3] \\
\end{bmatrix}
= \begin{bmatrix}
\nu[\alpha_1, \beta_1] \\
\nu[\alpha_2, \beta_2] \\
\nu[\alpha_3, \beta_3] \\
\end{bmatrix}= \nu.\]
\end{example}

\begin{proof}[Proof of Proposition \ref{prop:allviablez}]
 Let $z$ be viable with corresponding decomposition $T(z) = \Theta_1 \oplus \cdots \oplus \Theta_r$.  By Remark \ref{rem:viable}(2), $\Theta_k = T_k(z)$ if $\alpha_k$ is odd and $\Theta_k = T_k(z)^t$ if $\alpha_k$ is even.
 
 Recalling Definition \ref{def:truncmunu}, we have 
  \[\mu[\alpha_k, \beta_k] = \left(\mu_{\left\lceil\frac{\alpha_k + 1}{2}\right\rceil}, \dots, \mu_{\left\lceil\frac{\beta_k + 1}{2}\right\rceil}\right)^t,\]
  \[\nu[\alpha_k, \beta_k] = \left(\nu_{\left\lceil \frac{\alpha_k}{2} \right\rceil}, \dots, \nu_{\left\lceil \frac{\beta_k}{2} \right\rceil} \right)^t.\]
  Note that $\mu[\alpha_k, \beta_k]$ and $\nu[\alpha_k, \beta_k]$ partition $\mu$ and $\nu$ respectively to conform with the decomposition of $T(z) = \Theta_1 \oplus \cdots \oplus \Theta_r$ so that
 $$T(z)^t \mu = \begin{bmatrix} \Theta_1^t \mu[\alpha_1, \beta_1] \\ \vdots \\ \Theta_r^t \mu[\alpha_r, \beta_r] \end{bmatrix} \ \text{ and } \; \ T(z) \nu = \begin{bmatrix} \Theta_1 \nu[\alpha_1, \beta_1] \\ \vdots \\ \Theta_r \nu[\alpha_r, \beta_r] \end{bmatrix}.$$  
 It hence suffices to prove for each $k$ that $\mu[\alpha_k, \beta_k]$ is in the $1$-eigenspace of $\Theta_k \Theta_k^t$ and $\nu[\alpha_k, \beta_k]$ is in the $1$-eigenspace of $\Theta_k^t \Theta_k.$  
 
 By Definition \ref{D:viable}, we have $T_k(z)y_k = x_k$ and $T_k(z)^t x_k = y_k$, where $x_k = \mu(f_k)/\|\mu(f_k)\|$ and $y_k = \nu(f_k)/\|\nu(f_k)\|$.  This implies
 \[ T_k(z)T_k(z)^t x_k = x_k, \qquad T_k(z)^tT_k(z) y_k = y_k\]
 and hence
 \[ T_k(z)T_k(z)^t \mu(f_k) = \mu(f_k), \qquad T_k(z)^tT_k(z) \nu(f_k) = \nu(f_k).\]  
 Recall that $\Theta_k = T_k(z)$ if $\alpha_k$ is odd and $\Theta_k = T_k(z)^t$ if $\alpha_k$ is even.
 
 Consider first the case when $\alpha_k$ is odd and $\Theta_k = T_k(z)$.   We have 
 \[ \mu(f_k) = \left(1, \frac{d_{\alpha_k}}{d_{\alpha_k + 1}}, \dots \right)^t,
 \qquad \nu(f_k) = \left(1, \frac{d_{\alpha_k + 1}}{d_{\alpha_k + 2}}, \dots \right)^t,\]
 and, since $\alpha_k$ is odd, by the recursive definitions of $\mu$  and $\nu$, we have that 
 \[ \mu[\alpha_k, \beta_k] = \left(\mu_{\frac{\alpha_k + 1}{2}}, \dots, \mu_{\left\lceil\frac{\beta_k + 1}{2}\right\rceil}\right)^t = \mu_{\frac{\alpha_k + 1}{2}} \cdot \mu(f_k),\]
 \[\nu[\alpha_k, \beta_k] = \left(\nu_{\frac{\alpha_k +1}{2}}, \dots, \nu_{\left\lceil \frac{\beta_k}{2} \right\rceil}\right)^t = \nu_{\frac{\alpha_k + 1}{2}} \cdot \nu(f_k),\]
and hence $\mu[\alpha_k, \beta_k]$ is in the $1$-eigenspace for $\Theta_k \Theta_k^t = T_k(z)T_k(z)^t$ and $\nu[\alpha_k,\beta_k]$ is in the $1$-eigenspace for $\Theta_k^t \Theta_k = T_k(z)^t T_k(z)$.  

 Next, consider the case when $\alpha_k$ is even and $\Theta_k = T_k(z)^t$.  Since 
 \[ \nu(f_k) = \left(1, \frac{d_{\alpha_k + 1}}{d_{\alpha_k + 2}}, \dots \right)^t,\qquad  
 \mu(f_k) = \left(1, \frac{d_{\alpha_k}}{d_{\alpha_k + 1}}, \dots \right)^t,\]
 and $\alpha_k$ is even we have
 \[ \mu[\alpha_k, \beta_k] = \left(\mu_{\frac{\alpha_k}{2} + 1}, \dots, \mu_{\left\lceil\frac{\beta_k + 1}{2}\right\rceil}\right)^t = \mu_{\frac{\alpha_k}{2} + 1} \cdot \nu(f_k),\] \[\nu[\alpha_k, \beta_k] = \left(\nu_{\frac{\alpha_k}{2}}, \dots, \nu_{\left\lceil \frac{\beta_k}{2} \right\rceil}\right)^t = \nu_{\frac{\alpha_k}{2}} \cdot \mu(f_k).\]  Thus, when $\alpha_k$ is even, $\mu[\alpha_k, \beta_k]$ is in the $1$-eigenspace for $\Theta_k \Theta_k^t = T_k(z)^tT_k(z)$ and $\nu[\alpha_k,\beta_k]$ is in the $1$-eigenspace for $\Theta_k^t \Theta_k = T_k(z)T_k(z)^t$.  Thus, in any case, $\mu[\alpha_k, \beta_k]$ is in the $1$-eigenspace of $\Theta_k \Theta_k^t$ and $\nu[\alpha_k, \beta_k]$ is in the $1$-eigenspace of $\Theta_k^t \Theta_k$, completing the proof.
\end{proof}

\begin{remark}
 The entries of $T(z)^tT(z)$ and $T(z)T(z)^t$ are all sums of products of at most two $z_i$, where $z$ is treated as a variable for a fixed $d$.  Thus, the system of equations produced by Proposition \ref{prop:allviablez} is at most quadratic in the entries of $z$ and can be solved by, e.g., Mathematica \cite{Mathematica} relatively quickly.  (See Example \ref{ex:6} for such code.)
\end{remark}

Much like in the case when we can find a viable $z$ with all positive entries, there are cases when we can guarantee that certain viable $z$ are the best possible.

\begin{definition}
    \label{def:xy}
 Let $\mu$, $\nu$ be as in Definition \ref{def:munu}.  Define
 \[ x \colonequals \frac{\mu}{\|\mu\|}\qquad\text{and}\qquad y \colonequals \frac{\nu}{\|\nu\|}.\]
\end{definition}

\begin{prop}
 \label{prop:guaranteedbest}
 Let $x$ and $y$ be as in Definition \ref{def:xy}.  There exists a viable $z \in \IR^{n-1}$ such that $T(z)y = x$ and $T(z)^t x = y$ if and only if 
 \[ d^D(1,n) = \frac{\|\mu\| \|\nu\|}{d_1}.\]
\end{prop}

\begin{proof}
 Assume first that such a viable $z$ exists.  By Lemma \ref{lem:muTnu}, we have
 \[ x^t T(z) y = \frac{d_1}{\|\mu\| \|\nu\|} \cdot \sum_{i=1}^{n-1} z_i,\]
 and, by hypothesis, 
 \[ x^t T(z) y = \left(T(z)^t x\right)^t y = y^t y = 1,\]
 which implies 
 \[ \sum_{i = 1}^{n-1} z_i = \frac{\|\mu\| \|\nu\|}{d_1}.\]
 By Lemma \ref{lem:prerefine} and Theorem \ref{thm:maxatviabledecomp}, this is the largest that $\sum_{i = 1}^{n-1} z_i$ can be.  
 
 Conversely, if $d^D(1,n) = \|\mu\| \|\nu\|/d_1$, then this maximum is attained at some viable $z$ by Theorem \ref{thm:maxatviabledecomp}.  Using Lemma \ref{lem:muTnu}, since we know
 \[ x^t T(z) y = \frac{d_1}{\|\mu\| \|\nu\|} \cdot \sum_{i=1}^{n-1} z_i = \frac{d_1}{\|\mu\| \|\nu\|} \cdot d^D(1,n) = 1,\]
 we have 
 \[x^t(T(z)y - x) = x^tT(z)y - x^t x = 0\]
 and, similarly,
 \[ y^t(T(z)^t x - y) = 0, \]
 since $x$ and $y$ are unit vectors.  Thus, there exist vectors $x^\perp$ and $y^\perp$ such that $x^t x^\perp = y^t y^\perp = 0$ and $T(z)y = x + x^\perp$, $T(z)^t x = y + y^\perp$.  Furthermore, since $x = \mu/ \|\mu\|$, $y = \nu/\|\nu\|$, and $z$ is viable, by Proposition \ref{prop:allviablez}, we have
 \[ y = T(z)^tT(z)y = T(z)^tx + T(z)^tx^\perp = y + y^\perp + T(z)^t x^\perp,\]
 so $T(z)^t x^\perp = -y^\perp$, and a similar calculation shows that $T(z) y^\perp = -x^\perp$. So,
 \[ \|y^\perp\|^2 = \left(y + y^\perp \right)^t y^\perp = \left(T(z)^t x \right)^t y^\perp = x^tT(z)y^\perp = x^t(-x^\perp ) = 0,\]
and so $y^\perp = 0$.  Moreover, $x^\perp = -T(z)y^\perp = 0$, and hence $T(z)y = x$ and $T(z)^t x = y$. 
 \end{proof}

Due to the recursive nature of the definitions of $\mu$ and $\nu$, it is possible to solve the linear system $T(z)y = x$ and $T(z)^t x = y$, and there is a unique such $z \in \IR^{n-1}$ satisfying these equations.  In order to state this solution, we introduce some notation.

\begin{definition}
 Following the notation of Definition \ref{def:munu}, for $1 \le k \le \lfloor \frac{n+1}{2} \rfloor$, define
 \label{def:partialmunu}
 \[ \mu^{(k)} \colonequals (\mu_1, \dots, \mu_k)^t,\]
 and, for $1 \le k \le n - \lfloor \frac{n+1}{2} \rfloor$, define
 \[ \nu^{(k)} \colonequals (\nu_1, \dots, \nu_k)^t.\]
\end{definition}

\begin{cor}
 \label{cor:guaranteedbest}
Let $D \in M_n(\IR)$ be a symmetric tridiagonal matrix with zeros on the diagonal and $d_1, \dots, d_{n-1} > 0$ on the superdiagonal.  Then, 
\[ d^D(1,n) = \frac{\|\mu\| \|\nu\|}{d_1}\]
if and only if the vector $z$ is viable, where 
\[z_1 = \frac{\|\nu\|}{d_1 \|\mu\|}, \]
and, for $i \ge 1$,
\[ z_{2i} = \frac{1}{d_1 \|\mu\| \|\nu\|}\left( \|\nu^{(i)}\|^2\|\mu\|^2 - \|\mu^{(i)}\|^2 \|\nu\|^2\right),\]
\[ z_{2i+1} = \frac{1}{d_1 \|\mu\| \|\nu\|}\left( \|\mu^{(i+1)}\|^2\|\nu\|^2 - \|\nu^{(i)}\|^2 \|\mu\|^2\right).\]
\end{cor}

\begin{proof}
 This follows from Proposition \ref{prop:guaranteedbest} and solving the linear system $T(z)y = x$ and $T(z)^tx = y$ for $z$.
\end{proof}

\begin{remark}
The $z$ listed in Corollary \ref{cor:guaranteedbest} is not guaranteed to be viable; for example, if $\|\nu\| > \|\mu\|$, then $z_2 < 0$.  
\end{remark}

In the case all the weights on the edges of a path are equal, we can recover the result obtained in \cite{lattice_96_Atzmon}; see \cite[Proposition 3]{Besnard}.

\begin{corollary} \label{cor:standardpath} Suppose $d = (1, \dots, 1) \in \IR^{n-1}$ and $z$ is a viable vector such that the maximum is attained at $z$.  Then one of the following holds.
\begin{itemize}
    \item[(i)] $n = 2k$, $z_1+ \cdots+z_{n-1} = k$ is attained when
    $z_j = 1$ if $j$ is odd, and $z_j = 0$, otherwise.
\item[(i)] $n = 2k-1$, $z_1+ \cdots+z_{n-1} = \sqrt{k(k-1)}$ is attained when
\[z_{2j - 1} = j\sqrt{\frac{k-1}{k}} - (j-1)\sqrt{\frac{k}{k-1}}, \quad 
z_{2j} = j\left(\sqrt{\frac{k}{k-1}} - \sqrt{\frac{k-1}{k}}\right), \quad
 \quad 1 \le j \le \lfloor n/2 \rfloor.\]
\end{itemize}
\end{corollary}

\begin{proof}
In each case, the proposed $z$ satisfies the hypotheses of Proposition \ref{prop:guaranteedbest}. 
\end{proof}

\section{\texorpdfstring{Analysis for Paths of Length $\leq6$ and Outlook}{Analysis for Paths of Length at Most 6 and Outlook}}
\label{sect:analysis}

In this section, we include an analysis of paths with $n$ vertices, where $n \le 6$.  If $n \le 3$, we can always find a $z$ that satisfies the hypotheses of Proposition \ref{prop:guaranteedbest}, so we have the following.

\begin{corollary} 
\begin{itemize}
\item[{\rm (i)}]If $n = 2$,  then the maximum $d^D(1,2) = 1/d_{1}$ is attained at $z = 1/d_1$.
\item[{\rm (ii)}]If $n = 3$, then the maximum 
\[d^D(1,3) = \frac{1}{d_1}\sqrt{1+ \left(\frac{d_1}{d_{2}}\right)^2}\]
is attained at \[z = \frac{1}{d_1 \|\mu\| \|\nu\|}\left(\|\nu\|^2, \|\mu\|^2 - \|\nu\|^2 \right) = \frac{1}{d_1\sqrt{1 + \frac{d_1^2}{d_2^2}}}\left(1, \frac{d_1^2}{d_2^2}\right) = 
\frac{1}{\sqrt{d_{1}^2+d_{2}^2}}
\left(\frac{d_{2}}{d_{1}}, \frac{d_{1}}{d_{2}}\right).\]
\end{itemize}
\end{corollary}

When $n = 4$, the situation begins to grow more complicated; see \cite[Proposition 3]{Besnard}.

\begin{cor}
 \label{cor:4path}
 Let $D \in M_4(\IR)$ be a symmetric tridiagonal matrix with zeros on the diagonal and $d_1, d_2, d_3$ on the superdiagonal.  Then one of the following holds.
 \begin{itemize}
  \item[(a)] If $\|\nu \| > \|\mu \|$ (or, equivalently, if $d_2^2 > d_1 d_3$), then the maximum
  \[ d^D(1,4) = \frac{1}{d_1} + \frac{1}{d_3}\]
  is attained at $z = (1/d_1, 0, 1/d_3)$.
  \item[(b)] If $\|\mu\| \ge \|\nu\|$ (or, equivalently, $d_2^2 \le d_1 d_3$), then the maximum 
  \[ d^D(1,4) = \frac{\|\mu\| \|\nu\|}{d_1}= \frac{1}{d_1}\sqrt{1 +\left(\frac{d_1}{d_2} \right)^2} \sqrt{1 + \left(\frac{d_2}{d_3} \right)^2}\]
  is attained at 
  \[ z = \frac{1}{d_1 \|\mu\| \|\nu\|}\left(\|\nu\|^2,\|\mu\|^2 - \|\nu\|^2 ,\|\mu\|^2 \|\nu\|^2 - \|\mu\|^2 \right).\]
 \end{itemize}
\end{cor}

\begin{proof}
 By Corollary \ref{cor:guaranteedbest}, the $z$ listed in case (b) is the unique solution to the linear system  $T(z)y=x$ and $T(z)^t x=y$, and this $z$ is viable precisely when its entries are nonnegative.  Since $\|\nu\| > 1$, the third coordinate is always positive, and the second coordinate is nonnegative 
 precisely when $\|\mu\| \ge \|\nu\|$ (or, equivalently, when $d_2^2 \le d_1d_3$).
 
 In the event that $\|\mu\| < \|\nu\|$ (or, equivalently, $d_2^2 > d_1d_3$), there is no viable $z$ with the trivial decomposition by Proposition \ref{prop:guaranteedbest} and Corollary \ref{cor:guaranteedbest}.  The only possible way to obtain a viable $z$ with a nontrivial decomposition is with $d_2 = 0$, in which case the result follows by Theorem \ref{thm:maxatviabledecomp} (or Proposition \ref{prop:allviablez}).
\end{proof}

For example, if $n = 4$, $(d_{1}, d_{2}, d_{3}) = (3,2,1)$, then 
 the optimal $a_1-a_4$ should be $1/3 + 1$ by Corollary \ref{cor:4path}.  The value from (b), $(1/3)\sqrt{1+(3/2)^2}\sqrt{1+2^2}$, is larger than $1/3 + 1$,
but it is not attainable, since there is no viable $z$.

We can also deduce the following result when $n=5$. 

\begin{cor}
 \label{cor:5path}
 Let $D \in M_5(\IR)$ be a symmetric tridiagonal matrix with zeros on the diagonal and $d_1, d_2, d_3, d_4$ on the superdiagonal.  Then one of the following holds.
 \begin{itemize}
  \item[(a)] If $\|\nu\| > \|\mu\|$, then the maximum
  \[ d^D(1,5) = \frac{1}{d_1} + \frac{1}{d_3}\sqrt{1 + \left(\frac{d_3}{d_4}\right)^2}\] is attained at
  \[ z = \left(\frac{1}{d_1}, 0, \frac{d_4}{d_3\sqrt{d_3^2 + d_4^2}}, \frac{d_3}{d_4\sqrt{d_3^2 + d_4^2}} \right).\]
  \item[(b)] If $\|\mu\| > \sqrt{1 + \frac{d_1^2}{d_2^2}}\cdot \|\nu\|$, then the maximum
  \[ d^D(1,5) = \frac{1}{d_4} + \frac{1}{d_1} \sqrt{1 + \left(\frac{d_1}{d_2} \right)^2}\]
  is attained at
  \[ z = \left(\frac{d_2}{d_1\sqrt{d_1^2 + d_2^2}}, \frac{d_1}{d_2\sqrt{d_1^2 + d_2^2}}, 0, \frac{1}{d_4} \right).\]
  \item[(c)] If $\| \nu \| \le \| \mu \| \le \sqrt{1 + \frac{d_1^2}{d_2^2}}\cdot \|\nu\|$, then the maximum 
  \[ d^D(1,5) = \frac{\|\mu\| \|\nu\|}{d_1} = \frac{1}{d_1} \sqrt{1 + \left(\frac{d_1}{d_2} \right)^2 + \left(\frac{d_1d_3}{d_2d_4} \right)^2} \sqrt{1 + \left(\frac{d_2}{d_3}\right)^2}\]
  occurs at 
  \[ z = \frac{1}{d_1\|\mu\|\|\nu\|}\left(\|\nu\|^2,\|\mu\|^2 - \|\nu\|^2 ,\left(1 + \frac{d_1^2}{d_2^2} \right) \|\nu\|^2 - \|\mu\|^2,  \|\nu\|^2 \|\mu\|^2 - \left(1 + \frac{d_1^2}{d_2^2} \right)\|\nu\|^2\right).\]
 \end{itemize}
\end{cor}

\begin{proof}
 The proof is analogous to that of Corollary \ref{cor:4path}: by Corollary \ref{cor:guaranteedbest}, if the vector $z$ listed in (c) is viable, then the maximum occurs there.  This vector is viable precisely when its entries are nonnegative, i.e., when 
 \[ \| \nu \| \le \| \mu \| \le \sqrt{1 + \frac{d_1^2}{d_2^2}}\cdot \|\nu\|.\]
 At most one of these inequalities can fail, corresponding to (a) and (b), respectively.
\end{proof}

Finally, we end with a brief discussion of the case when $n = 6$.  While the formula for the maximum distance will be one of the formulas from Theorem \ref{thm:maxatviabledecomp}, it is less clear in this instance which is best.  Indeed, when $n \ge 6$, it is possible that multiple coordinates from the $z$ obtained in Corollary \ref{cor:guaranteedbest} fail to be nonnegative simultaneously, making the situation far more complicated.  For example, when the second and fourth coordinates of the $z$ from Corollary \ref{cor:guaranteedbest} are negative, the maximum value of $d^D(1,6)$ could be given by multiple formulas (see the following example).  Instead, we show in the following example that all possible ``zero patterns'' for viable $z$ actually occur as the best viable $z$ in some instance; put differently, for every formula for $d^D(1,6)$ that is a possibility in Theorem \ref{thm:maxatviabledecomp}, there is a $D$ for which that formula holds.

\begin{example}
\label{ex:6}
The interested reader can verify the examples below with the following Mathematica code to calculate $d^D(1, 6)$ once $d_1, d_2, d_3, d_4$, and $d_5$ are given.  Note that the sixth command (which begins with ``FullSimplify'') attempts to find the best possible solution by looking for a viable $z$ satisfying the hypotheses of Proposition \ref{prop:guaranteedbest}.  The final three commands utilize Proposition \ref{prop:allviablez} to find all viable $z$ and simply chooses the best one.

\begin{verbatim}
 mu = {{1}, {d1/d2}, {d1*d3/(d2*d4)}}
 
 x = mu/Norm[mu]
 
 nu = {{1}, {d2/d3}, {(d2*d4)/(d3*d5)}}
 
 y = nu/Norm[nu]
 
 Tz = {{d1*z1, 0, 0}, {d2*z2, d3*z3, 0}, {0, d4*z4, d5*z5}}
 
 FullSimplify[
 Solve[Transpose[Tz] . x == y && Tz . y == x  && 
   z1 \[GreaterSlantEqual] 0 && z2 \[GreaterSlantEqual] 0 && 
   z3 \[GreaterSlantEqual] 0 && z4 \[GreaterSlantEqual] 0 && 
   z5 \[GreaterSlantEqual] 0, {z1, z2, z3, z4, z5}]]
    
 l = Solve[
  Tz . Transpose[Tz] . mu == mu && Transpose[Tz] . Tz . nu == nu  && 
   z1 \[GreaterSlantEqual] 0 && z2 \[GreaterSlantEqual] 0 && 
   z3 \[GreaterSlantEqual] 0 && z4 \[GreaterSlantEqual] 0 && 
   z5 \[GreaterSlantEqual] 0, {z1, z2, z3, z4, z5}]
       
       
 Total[{z1, z2, z3, z4, z5} /. Part[l, #]] & /@  Range[Length[l]]

 Max[Total[{z1, z2, z3, z4, z5} /. Part[l, #]] & /@  Range[Length[l]]]
\end{verbatim}

 \begin{itemize}
  \item[(a)] When $d = (2, 1, 1, 1, 1)$, the maximum
  \[ d^D(1,6) = \frac{\|\mu\| \|\nu\|}{d_1} = \frac{1}{d_1} \sqrt{1 + \left(\frac{d_1}{d_2} \right)^2 + \left(\frac{d_1d_3}{d_2d_4} \right)^2} \sqrt{1 + \left(\frac{d_2}{d_3}\right)^2 + \left(\frac{d_2d_4}{d_3d_5} \right)^2} = \frac{3\sqrt{3}}{2}\]
  is attained at
  \[ z = \left(\frac{1}{2\sqrt{3}}, \frac{1}{\sqrt{3}}, \frac{1}{\sqrt{3}}, \frac{1}{2\sqrt{3}}, \frac{\sqrt{3}}{2} \right),\]
  which is the solution to $T(z)y = x$ and $T(z)^t x = y$ from Corollary \ref{cor:guaranteedbest}.
  \item[(b)] When $d = (1, 2, 1, 2, 1)$, the maximum
  \[ d^D(1,6) = \frac{1}{d_1} + \frac{1}{d_3} + \frac{1}{d_5} = 3\]
  is attained at 
  \[ z = (1, 0, 1, 0, 1).\]
  Note that the solution to $T(z)y = x$ and $T(z)^tx = y$ from Corollary \ref{cor:guaranteedbest} in this instance is 
  \[ z = \left(4, -\frac{15}{4}, \frac{19}{4}, -\frac{15}{4}, 4 \right),\]
  which is negative in the second and fourth coordinates.
  \item[(c)] When $d = (3, 20, 100, 10, 1000)$, the maximum
  \[ d^D(1,6) = \frac{1}{d_1}\sqrt{1 + \left(\frac{d_1}{d_2}\right)^2} + \frac{1}{d_4}\sqrt{1 + \left( \frac{d_4}{d_5}\right)^2} = \frac{\sqrt{409}}{60} + \frac{\sqrt{10001}}{1000}\]
  is attained at 
  \[ z = \left(\frac{20}{3\sqrt{409}}, \frac{3}{20\sqrt{409}}, 0, \frac{10}{\sqrt{10001}}, \frac{1}{1000\sqrt{10001}} \right).\]
  Note that the solution to $T(z)y = x$ and $T(z)^tx = y$ from Corollary \ref{cor:guaranteedbest} in this instance is 
  \[ z = \left(\frac{\sqrt{\frac{4127}{187}}}{25}, \frac{79732}{225\sqrt{771749}}, -\frac{31558513}{90000\sqrt{771749}}, \frac{33428513}{90000\sqrt{771749}}, \frac{\sqrt{\frac{187}{4127}}}{90000} \right),\]
  which is negative in the third coordinate.
  \item[(d)] When $d = (1, 3, 2, 1, 1)$, the maximum
  \[ d^D(1,6) = \frac{1}{d_1} +  \frac{1}{d_3}\sqrt{1 +\left(\frac{d_3}{d_4} \right)^2} \sqrt{1 + \left(\frac{d_4}{d_5} \right)^2} = 1 + \sqrt{\frac{5}{2}}\] is attained at
  \[ z = \left(1, 0, \frac{1}{\sqrt{10}}, \frac{3}{2\sqrt{10}}, \frac{\sqrt{\frac{5}{2}}}{2} \right).\]
  Note that the solution to $T(z)y = x$ and $T(z)^tx = y$ from Corollary \ref{cor:guaranteedbest} in this instance is 
  \[ z = \left(\frac{3 \sqrt{\frac{11}{7}}}{2}, -\frac{71}{6\sqrt{77}}, \frac{41}{3\sqrt{77}}, -\frac{19}{6\sqrt{77}}, \frac{3\sqrt{\frac{7}{11}}}{2} \right),\]
  which is negative in the second and fourth coordinates.
  \item[(e)] When $d = (1, 1, 2, 3, 1)$, the maximum
  \[ d^D(1,6) = \frac{1}{d_1}\sqrt{1 +\left(\frac{d_1}{d_2} \right)^2} \sqrt{1 + \left(\frac{d_2}{d_3} \right)^2} + \frac{1}{d_5} = 1 + \sqrt{\frac{5}{2}}\] is attained at
  \[ z = \left(\frac{\sqrt{\frac{5}{2}}}{2}, \frac{3}{2\sqrt{10}}, \frac{1}{\sqrt{10}}, 0 ,1 \right).\]
  Note that the solution to $T(z)y = x$ and $T(z)^tx = y$ from Corollary \ref{cor:guaranteedbest} in this instance is 
  \[ z = \left(\frac{3 \sqrt{\frac{7}{11}}}{2}, -\frac{19}{6\sqrt{77}}, \frac{41}{3\sqrt{77}}, -\frac{71}{6\sqrt{77}}, \frac{3\sqrt{\frac{11}{7}}}{2} \right),\]
  which is negative in the second and fourth coordinates.
 \end{itemize}
\end{example}

\begin{remark}
 It seems clear that the number of truly distinct subcases one must check increases in conjunction with the number of partitions of $n$.  Hence, the problem becomes computationally intractable very quickly without further simplifying assumptions. 
\end{remark}

\medskip\noindent
{\bf \large Acknowledgment}

Li is an affiliate member of the Institute for Quantum Computing, University of Waterloo.
His research was partially supported by the Simons Foundation Grant  851334.

\bibliographystyle{amsalpha}
\bibliography{NCDistance.bib}

\noindent
{\bf \large Addresses:}

\smallskip\noindent
(Clare, Li, Swartz) Department of Mathematics, College of William \& Mary, 
Williamsburg, VA 23185, USA. E-mail: prclare@wm.edu, easwartz@wm.edu,  ckli@math.wm.edu.

\smallskip\noindent
(Poon) Department of Mathematics, Embry-Riddle Aeronautical University, Prescott AZ 86301, USA.
E-mail: poon3de@erau.edu

\end{document}